\documentclass[a4paper,12pt]{article}
\usepackage{a4wide}
\usepackage{amsmath}
\usepackage{amssymb}
\usepackage{amsthm}
\usepackage{latexsym}
\usepackage{graphicx}
\usepackage[english]{babel}
\usepackage{makeidx}

\newtheorem{obs} [subsection]{Remark}
\newtheorem{exm} [subsection]{Example}

\newtheorem{teor}[subsection]{Theorem}
\newtheorem{lema}[subsection]{Lemma}

\newcommand{\Zng}{$\mathbb Z^n$-graded $S$-module}

\def\sdepth{\operatorname{sdepth}}
\def\depth{\operatorname{depth}}

\begin{document}
\selectlanguage{english}
\frenchspacing

\large
\begin{center}
\textbf{Stanley depth of the path ideal associated to a line graph}

Mircea Cimpoea\c s
\end{center}
\normalsize

\begin{abstract}
We consider the path ideal associated to a line graph, we compute \texttt{sdepth} for its quotient ring and note that it is equal with its \texttt{depth}. In particular, it satisfies the Stanley inequality.


\end{abstract}

\section*{Introduction}

Let $K$ be a field and $S=K[x_1,\ldots,x_n]$ the polynomial ring over $K$.
Let $M$ be a \Zng. A \emph{Stanley decomposition} of $M$ is a direct sum $\mathcal D: M = \bigoplus_{i=1}^rm_i K[Z_i]$ as a $\mathbb Z^n$-graded $K$-vector space, where $m_i\in M$ is homogeneous with respect to $\mathbb Z^n$-grading, $Z_i\subset\{x_1,\ldots,x_n\}$ such that $m_i K[Z_i] = \{um_i:\; u\in K[Z_i] \}\subset M$ is a free $K[Z_i]$-submodule of $M$. We define $\sdepth(\mathcal D)=\min_{i=1,\ldots,r} |Z_i|$ and $\sdepth_S(M)=\max\{\sdepth(\mathcal D)|\;\mathcal D$ is a Stanley decomposition of $M\}$. The number $\sdepth_S(M)$ is called the \emph{Stanley depth} of $M$. In \cite{apel}, J.\ Apel restated a conjecture firstly given by Stanley in \cite{stan}, namely that $\sdepth_S(M)\geq\depth_S(M)$ for any \Zng $\;M$. This conjecture proves to be false, in general, for $M=S/I$ and $M=J/I$, where $0\neq I\subset J\subset S$ are monomial ideals, see \cite{duval}.

Herzog, Vladoiu and Zheng show in \cite{hvz} that $\sdepth_S(M)$ can be computed in a finite number of steps if $M=I/J$, where $J\subset I\subset S$ are monomial ideals. In \cite{rin}, Rinaldo give a computer implementation for this algorithm, in the computer algebra system $CoCoA$ \cite{cocoa}. However, it is difficult to compute this invariant, even in some very particular cases.  For instance in \cite{par} Biro et al. proved that $\sdepth(\mathbf m)= \left\lceil n/2 \right\rceil$ where $\mathbf m=(x_1,\ldots,x_n)$. For a friendly introduction on Stanley depth we recommend \cite{her}.



Let $\Delta \subset 2^{[n]}$ be a simplicial complex. A face $F\in\Delta$ is called a \emph{facet}, if $F$ is maximal with respect to inclusion. We denote $\mathcal F(\Delta)$ the set of facets of $\Delta$. If $F\in\mathcal F(\Delta)$, we denote $x_F=\prod_{j\in F}x_j$. Then the \emph{facet ideal $I(\Delta)$} associated to $\Delta$ is the squarefree monomial ideal $I=(x_F\;:\;F\in \mathcal F(\Delta))$ of $S$. The facet ideal was studied by Faridi \cite{faridi} from the \texttt{depth} perspective.

A line graph of lenght $n$, denoted by $L_n$, is a graph with the vertex set $V=[n]$ and the edge set $E=\{\{1,2\},\{2,3\},\ldots,\{n-1,n\}\}$. The Stanley depth of the edge ideal associated to $L_n$ (which is in fact the facet ideal of $L_n$, if we look at $L_n$ as a simplicial complex) was computed by Alin \c Stefan in \cite{alin}. 

\footnotetext[1]{We greatfully acknowledge the use of the computer algebra system CoCoA (\cite{cocoa}) for our experiments.}
\footnotetext[2]{The support from grant ID-PCE-2011-1023 of Romanian Ministry of Education, Research and Innovation is gratefully acknowledged.}

Let $\Delta_{n,m}$ be the simplicial complex with the set of facets 
$\mathcal F(\Delta_{n,m})=\{\{1,2,\ldots,m\}, \linebreak, \{2,3,\ldots,m+1\},\cdots, \{n-m+1,n-m+2,\ldots,n\}\}$. We denote $I_{n,m} =(x_1x_2\cdots x_m, \linebreak x_2x_3\cdots x_{m+1}, \ldots, x_{n-m+1}x_{n-m+2}\cdots x_n )$
, the associated facet ideal. 

Note that $I_{n,m}$ is the path ideal of the graph $L_n$, provided with the direction given by $1<2<\ldots <n$, see \cite{tuy} for further details.

According to \cite[Theorem 1.2]{tuy}, 
$$pd(S/I_{n,m}) = \begin{cases} \frac{2(n-d)}{m+1},\; n\equiv d (mod\;(m+1))\;with\; 0 \leq d\leq m-1, \\ 
\frac{2n-m+1}{m+1},\; n\equiv m (mod\;(m+1)).
 \end{cases}$$ 
By Auslander-Buchsbaum formula (see \cite{real}), it follows that $\depth(S/I_{n,m})=n-pd(S/I_{n,m})$ and, by a straightforward computation, we can see $\depth(S/I_{n,m}) =  n+1 - \left\lfloor \frac{n+1}{m+1} \right\rfloor - \left\lceil \frac{n+1}{m+1} \right\rceil$. 

We prove that $\sdepth(S/I_{n,m})=\depth(S/I_{n,m}) = n+1 - \left\lfloor \frac{n+1}{m+1} \right\rfloor - \left\lceil \frac{n+1}{m+1} \right\rceil$, see Theorem $1.3$. In particular, we give another prove for the result of \cite[Theorem 1.2]{tuy}. Also, our result generalize \cite[Lemma 4]{alin}.

We recall some notions introduced by Faridi in \cite{faridi}. Let $\Delta$ be a simplicial complex. A facet $F$ of $\Delta$ is called a \emph{leaf}, if either $F$ is the only facet of $\Delta$, or there exists a facet $G$ in $\Delta$, $G\neq F$, such that $F\cap F' \subseteq F\cap G$ for all $F'\in \Delta$ with $F'\neq F$. A connected simplicial complex $\Delta$ is called a \emph{tree}, if every nonempty connected subcomplex of $\Delta$ has a leaf. This notion generalize trees from graph theory. Note that $\Delta_{n,m}$ is a tree, in the sense of the above definition.
 
According to \cite[Corollary 1.6]{jah}, if $I$ is the facet ideal associated to a tree (which is the case for $I_{n,m}$), it follows that $S/I$ would be pretty clean. However, there is a mistake in the second line of the proof of \cite[Proposition 1.4]{jah}, and therefore, this result might be wrong in general. On the other hand, if $I\subset S$ is a pretty clean monomial ideal, it is known that $\sdepth(S/I)= \depth(S/I)$, see \cite[Proposition 18]{her} for further details. 


\section{Main results}

We recall the well known Depth Lemma, see for instance \cite[Lemma 1.3.9]{real} or \cite[Lemma 3.1.4]{vasc}.

\begin{lema}(Depth Lemma)
If $0 \rightarrow U \rightarrow M \rightarrow N \rightarrow 0$ is a short exact sequence of modules over a local ring $S$, or a Noetherian graded ring with $S_0$ local, then

a) $\depth M \geq \min\{\depth N,\depth U\}$.

b) $\depth U \geq \min\{\depth M,\depth N +1 \}$.

c) $\depth N\geq \min\{\depth U - 1,\depth M\}$.
\end{lema}

In \cite{asia}, Asia Rauf proved the analog of Lemma $1.1(a)$ for $\sdepth$:

\begin{lema}
Let $0 \rightarrow U \rightarrow M \rightarrow N \rightarrow 0$ be a short exact sequence of $\mathbb Z^n$-graded $S$-modules. Then:
\[ \sdepth(M) \geq \min\{\sdepth(U),\sdepth(N) \}. \]
\end{lema}

Our main result is the following theorem.\pagebreak

\begin{teor}
$\sdepth(S/I_{n,m}) = \depth(S/I_{n,m})=n+1 - \left\lfloor \frac{n+1}{m+1} \right\rfloor - \left\lceil \frac{n+1}{m+1} \right\rceil$.
\end{teor}

\begin{proof}
We use induction on $m\geq 1$ and $n\geq m$. The case $m=1$ is trivial. The case $m=2$ follows from \cite[Lemma 2.8]{mor} and \cite[Lemma 4]{alin}. 

We assume $m\geq 3$. If $n=m$, then $\sdepth(S/I_{n,m})=\depth(S/I_{n,m})=m-1$, since $I_{n,n}=(x_1\cdots x_n)$ is principal. Assume $m+1\leq n\leq 2m-1$. Note that $I_{n,m}=x_{m}(I_{n,m}:x_m)$. We have $\sdepth(S/I_{n,m})=\sdepth(S/(I_{n,m}:x_m))$, by \cite[Theorem 1.4]{mir}. Also, we obviously have $\depth(S/I_{n,m})=\depth(S/(I_{n,m}:x_m))$. On the other hand, $S/(I_{n,m}:x_m)$ is isomorphic to $S'/(I_{n-1,m-1})[y]$, where $S'=K[x_1,\ldots,x_{m-1},x_{m+1},\ldots,x_n]$ and therefore, by induction hypothesis and \cite[Lemma 3.6]{hvz}, we get $\sdepth(S/I_{n,m})= \depth(S/I_{n,m}) = 1 + ( n - \left\lfloor \frac{n}{m} \right\rfloor - \left\lceil \frac{n}{m} \right\rceil ) = 1 + n - 3 = n-2$, as required.

It remains to consider the case $m\geq 3$ and $n\geq 2m$. 
Let $k:=\left\lfloor \frac{n+1}{m+1} \right\rfloor$ and $a=n+1-k(m+1)$. We denote $\varphi(n,m):=n+1 - \left\lfloor \frac{n+1}{m+1} \right\rfloor - \left\lceil \frac{n+1}{m+1} \right\rceil$. One can easily see that $\varphi(n,m)= \begin{cases} n+1-2k,\;a=0 \\ n-2k,\;a\neq 0  \end{cases}$.

We consider the ideals $L_0:=I_{n,m}$ and $L_{j}:=(L_{j-1}:x_{j(m+1)-1})$, where $1\leq j\leq k$. We denote $U_j:=(L_{j-1},x_{j(m+1)-1})$ for all $1\leq j \leq k$. We have the following short exact sequences:

\[ (\mathcal S_k):\;\; 0 \longrightarrow S/L_j \stackrel{\cdot x_{j(m+1)-1}}{\longrightarrow} S/L_{j-1} \longrightarrow S/U_j \longrightarrow 0,\;\;1\leq j\leq k. \]

We denote $u_i:=x_i\cdots x_{i+m-1}$, for $1\leq i\leq n-m+1$. Note that $G(L_0)=\{ u_1,\ldots,u_{n-m+1} \}$, $G(L_1)=\{ \frac{u_1}{x_m},\ldots,\frac{u_m}{x_m},  u_{m+2},\ldots,u_{n-m+1} \}$, because $u_{m+1}\in (u_m/x_m)$, and, also, \linebreak $G(U_1)=\{ x_m,u_{m+1},\ldots, u_{n-m+1} \}$. Moreover, one can easily check that: 
$$L_j = (\frac{u_1}{x_m},\ldots,\frac{u_m}{x_m}, \frac{u_{m+2}}{x_{2m+1}}, \ldots, \frac{u_{2m+1}}{x_{2m+1}}, \ldots, \frac{u_{(m+1)j-m}}{x_{(m+1)j-1}},\ldots, \frac{u_{(m+1)j-1}}{x_{(m+1)j-1}}, u_{(m+1)j+1},\ldots, u_{n-m+1} ),$$ 

for all $1\leq j\leq k-1$. It follows that:
$$U_{j+1} = (\frac{u_1}{x_m},\ldots,\frac{u_m}{x_m}, 
\ldots, \frac{u_{(m+1)j-m}}{x_{(m+1)j-1}},\ldots, \frac{u_{(m+1)j-1}}{x_{(m+1)j-1}}, x_{(m+1)(j+1)-1}, u_{(m+1)(j+1)},\ldots, u_{n-m+1}), $$
for all $1\leq j\leq k-1$. Also, we have:
$$L_k = (\frac{u_1}{x_m},\ldots,\frac{u_m}{x_m}, \ldots,
\frac{u_{(m+1)(k-1)-m}}{x_{(m+1)(k-1)-1}},\ldots, \frac{u_{(m+1)(k-1)-1}}{x_{(m+1)(k-1)-1}},
\frac{u_{(m+1)k-m}}{x_{(m+1)k-1}},\ldots, \frac{u_{t}}{x_{(m+1)k-1}}),$$
where $t=n-m$ if $a=m$, or $t=n-m+1$ otherwise.

Note that $|G(L_k)| = m(k-1) + (t+1) - (m+1)k + m = t + 1 - k$ and, moreover, $L_k \cong I_{t+m-k-1, m-1}S$. Thus, by induction hypothesis and \cite[Lemma 3.6]{hvz}, we have $\depth(S/L_k) = \sdepth(S/L_k) = n - (t+m-k-1) + \varphi(t+m-k-1, m-1) = n+1 - \left\lfloor \frac{t+m-k}{m} \right\rfloor - \left\lceil \frac{t+m-k}{m}\right\rceil$.

If $a=m$, then $t=n-m$, $n=k(m+1)+m-1$, $t+m-k = n-k = (k+1)m-1$ and thus $\depth(S/L_k) = \sdepth(S/L_k) = n + 1 - k-(k+1) = n-2k=\varphi(n,m)$. If $a=0$, then $t+m-k = km$ and thus $\depth(S/L_k) = \sdepth(S/L_k) = n + 1 - 2k$. 

If $0<a<m$, then $t+m-k = km+a$ and thus $\depth(S/L_k) = \sdepth(S/L_k)=n-2k$. In all the cases, we have $\depth(S/L_k) = \sdepth(S/L_k) = \varphi(n,m)$.

Note that $S/U_1 \cong K[x_{m+1},\ldots,x_n]/(u_{m+1},\ldots, u_{n-m+1}) [x_1,\ldots,x_{m-1}]$ and therefore, by induction hypothesis,
$\depth(S/U_1)=\sdepth(S/U_1) = m-1 + \varphi(n-m,m) = n - \left\lfloor \frac{n-m+1}{m+1} \right\rfloor - \left\lceil \frac{n-m+1}{m+1} \right\rceil$. Note that $\frac{n-m+1}{m+1} = k-1 + \frac{a+1}{m+1}$ and therefore $\left\lceil \frac{n-m+1}{m+1} \right\rceil = k$. On the other hand, if $a<m$ then $\left\lfloor \frac{n-m+1}{m+1} \right\rfloor = k-1$ and if $a=m$ then $\left\lfloor \frac{n-m+1}{m+1} \right\rfloor = k$. It follows that $\depth(S/U_1)=\sdepth(S/U_1) = \begin{cases} n+1-2k,\;a<m \\ n-2k,\; a=m \end{cases} \geq \varphi(n,m)$.

Moreover, $\depth(S/U_1)=\sdepth(S/U_1) =\varphi(n,m)$ if and only if $a=0$ or $a=m$. Otherwise, $\depth(S/U_1)=\sdepth(S/U_1) =\varphi(n,m)+1$. Assume $a=0$ or $a=m$. From the exact sequence $(\mathcal S_1) 0 \rightarrow S/L_1 \rightarrow S/L_0 \rightarrow S/U_1 \rightarrow 0$, Lemma $1.1$ and Lemma $1.2$, it follows that $\sdepth(S/L_0)\geq \depth(S/L_0)=\varphi(n,m)$. On the other hand, since $L_k=(L_0:x_mx_{2m+1}\cdots x_{k(m+1)-1})$, for example by \cite[Proposition 2.7]{mirci}, $\varphi(n,m) = \sdepth(S/L_k)\geq \sdepth(S/L_0) \geq \varphi(n,m)$. Thus, $\sdepth(S/L_k)=\varphi(n,m)$.

It remains to consider the case when $1<a<m-1$. We claim that:  $$(*) \sdepth(S/U_j)\geq \depth(S/U_j)\geq \varphi(n,m) \;\;\emph{for\;\;all}\;\; 2\leq j\leq k.$$ 

Assume this is the case. Using $1.1$, $1.2$ and the short exact sequences $(\mathcal S_k)$, we get, inductively, that $\sdepth(S/L_j)\geq \depth(S/L_j)=\varphi(n,m)$ for all $j<k-1$. 
Again, using for example \cite[Proposition 2.7]{mirci}, we get $\sdepth(S/L_0) = \varphi(n,m)$.

In order to complete the proof, we need to show $(*)$. Note that $U_k = (V_k,x_{(m+1)k-1})$, where
$V_k=(\frac{u_1}{x_m},\ldots,\frac{u_m}{x_m}, \ldots, \frac{u_{(m+1)j-m}}{x_{(m+1)j-1}},\ldots, \frac{u_{(m+1)(k-1)-1}}{x_{(m+1)(k-1)-1}})\cong I_{mk-2,m-1}S $. By induction hypothesis and \cite[Lemma 3.6]{hvz}, it follows that $\sdepth(S/U_k)=\depth(S/U_k)= n-(mk-2)-1 + \varphi(mk-2,m-1) = n - \left\lfloor \frac{mk-1}{m} \right\rfloor - \left\lceil \frac{mk-1}{m} \right\rceil = n - (k-1) - k = n-2k+1 = \varphi(n,m)+1$.

If $1\leq j <k$, we have $S/U_j \cong (S/V_j \otimes_S S/W_j S)/(x_{(m+1)j-1})(S/V_j \otimes_S S/W_j S)$, where $V_j=(\frac{u_1}{x_m},\ldots,\frac{u_m}{x_m}, \ldots, \frac{u_{(m+1)j-m}}{x_{(m+1)j-1}},\ldots, \frac{u_{(m+1)j-1}}{x_{(m+1)j-1}})$ and $W_j = (u_{(m+1)(j+1)},\ldots, u_{n-m+1})$. Since $x_{(m+1)j-1}$ is regular on $S/V_j \otimes_S S/W_j$ by \cite[Corollary 1.12]{asia} and \cite[Theorem 3.1]{asia} or \cite[Theorem 1.2]{mirci}, it follows that  $\depth(S/U_j)=\depth(S/V_j \otimes_S S/W_j)-1 = \depth(S/V_j) + \depth(S/W_j) - n - 1$ and $\sdepth(S/U_j)=\sdepth(S/V_j \otimes_S S/W_j)-1 \geq \sdepth(S/V_j) + \sdepth(S/W_j) - n - 1$. 

On the other hand, $V_j \cong I_{m(j+1)-2,m-1}S$ and thus, by induction hypothesis, 
$\sdepth(S/V_j)=\depth(S/V_j)=n+1 - \left\lfloor \frac{m(j+1)-1}{m} \right\rfloor - \left\lceil \frac{m(j+1)-1}{m} \right\rceil = n - 2j$. Also, $W_j \cong I_{n-(m+1)(j+1)+1,m}$ and, by induction hypothesis, we have
$\sdepth(S/W_j) = \depth(S/W_j) = n+1 - \left\lfloor \frac{n-(m+1)(j+1)+2}{m+1} \right\rfloor - \left\lceil \frac{n-(m+1)(j+1)+2}{m+1} \right\rceil = n+1 + 2(j+1) -\left\lfloor \frac{n+2}{m+1} \right\rfloor - \left\lceil \frac{n+2}{m+1} \right\rceil$.

It follows that $\sdepth(S/U_j)=\depth(S/U_j) = n+2 - \left\lfloor \frac{n+2}{m+1} \right\rfloor - \left\lceil \frac{n+2}{m+1} \right\rceil \geq \varphi(n,m)$, since either $\left\lfloor \frac{n+2}{m+1} \right\rfloor = \left\lfloor \frac{n+1}{m+1} \right\rfloor$ and $\left\lceil \frac{n+2}{m+1} \right\rceil = \left\lceil \frac{n+1}{m+1} \right\rceil$, either $\left\lfloor \frac{n+2}{m+1} \right\rfloor = \left\lfloor \frac{n+1}{m+1} \right\rfloor+1$ and $\left\lceil \frac{n+2}{m+1} \right\rceil = \left\lceil \frac{n+1}{m+1} \right\rceil$ or either $\left\lfloor \frac{n+2}{m+1} \right\rfloor = \left\lfloor \frac{n+1}{m+1} \right\rfloor$ and $\left\lceil \frac{n+2}{m+1} \right\rceil = \left\lceil \frac{n+1}{m+1} \right\rceil+1$.
\end{proof}

\pagebreak

\begin{exm}
\emph{Let $I_{6,3}=(x_1x_2x_3,x_2x_3x_4,x_3x_4x_5,x_4x_5x_6) \subset S:=K[x_1,\ldots,x_6]$. Note that $\varphi(7,4) = 7 - \left\lfloor \frac{7}{4}\right\rfloor  -\left\lceil \frac{7}{4}  \right\rceil = 4$. Let $L_0=I_{6,3}$, $L_1=(L_0:x_3)= (x_1x_2,x_2x_4,x_4x_5)$ and $U_1=(L_0,x_3) = (x_3,x_4x_5x_6)$. Since $L_1\cong I_{4,2}S$, it follows that $\depth(S/L_1) = \sdepth(S/L_1) = \depth(S/I_{4,2}S) = 2 + \depth(K[x_1,\ldots,x_4]/I_{4,2}) = 2+\varphi(4,2) = 4$.}

\emph{On the other hand, since $U_1$ is a complete intersection, $\depth(S/U_1)=\sdepth(S/U_1)=4$. We consider the short exact sequence $0 \rightarrow S/L_1 \rightarrow S/L_0 \rightarrow S/U_1 \rightarrow 0$. By Lemma $1.2$, it follows that $\sdepth(S/L_0)\geq 4$. On the other hand, since $L_1=(L_0:x_3)$, one has $\sdepth(S/L_0)\leq \sdepth(S/L_1) =4$. Thus $\sdepth(S/L_0)= 4$. Also, by Lemma $1.1$, $\depth(S/L_0) =4$.}
\end{exm}

In the following, we present another way to prove that $\sdepth(S/I_{n,m})\leq \varphi(n,m)$.

Let $\mathcal P\subset 2^{[n]}$ be a poset. If $C,D\subset [n]$, the \emph{interval} $[C,D]$ consist in all the subsets $X$ of $[n]$ such that $C\subset X\subset D$. Let $\mathbf P:\mathcal P=\bigcup_{i=1}^r [F_i,G_i]$ be a partition of $\mathcal P$, i.e. $[F_i,G_i]\cap [F_j,G_j] = \emptyset$ for all $i\neq j$. We denote $\sdepth(\mathbf P):=\min_{i\in [r]} |D_i|$. Also, we define the Stanley depth of $\mathcal P$, to be the number
\[\sdepth(\mathcal P) = \max\{\sdepth(\mathbf P):\; \mathbf P \; \emph{is\; a\; partition\; of} \; \mathcal P \}.\]
Now, for $d\in\mathbb N$ and $\sigma\in \mathcal P$, we denote
\[ \mathcal P_d = \{\tau\in\mathcal P\;:\; |\tau|=d \}\;,\;\mathcal P_{d,\sigma} = \{ \tau\in\mathcal P_d\;:\; \sigma\subset\tau \}. \]
Note that if $\sigma \in \mathcal P$ such that $P_{d,\sigma}=\emptyset$, then $\sdepth(\mathcal P)<d$. Indeed, let $\mathbf P:\mathcal P=\bigcup_{i=1}^r [F_i,G_i]$ be a partition of $\mathcal P$ with $\sdepth(\mathcal P)=\sdepth(\mathbf P)$. Since $\sigma\in\mathcal P$, it follows that  $\sigma \in [F_i,G_i]$ for some $i$. If $|G_i|\geq d$, then it follows that $\mathcal P_{d,\sigma}\neq\emptyset$, since there are subsets in the interval $[F_i,G_i]$ of cardinality $d$ which contain $\sigma$, a contradiction. Thus, $|G_i|<d$ and therefore $\sdepth(\mathcal P)<d$.

We recall the method of Herzog, Vladoiu and Zheng \cite{hvz} for computing the Stanley depth of $S/I$ and $I$, where $I$ is a squarefree monomial ideal. Let $G(I)=\{u_1,\ldots,u_s\}$ be the set of minimal monomial generators of $I$. We define the following two posets:
\[ \mathcal P_I:=\{\sigma \subset [n]:\; u_i|x_{\sigma}:=\prod_{j\in\sigma}x_j \;\emph{for\;some}\;i\;\}\;\emph{and}\; 
\mathcal P_{S/I}:=2^{[n]}\setminus \mathcal P_I. \]
Herzog Vladoiu and Zheng proved in \cite{hvz} that $\sdepth(I)=\sdepth(\mathcal P_{I})$ and $\sdepth(S/I)=\sdepth(\mathcal P_{S/I})$.

The above method is useful to give upper bounds for the $\sdepth(S/I)$, where $I\subset S$ is a monomial ideal, and, in particular cases, to compute the exact value of $\sdepth(S/I)$. That's exactly the case for $S/I_{n,m}$!

Let $\mathcal P:=\mathcal P_{S/I_{n,m}}$. We denote $k=\left\lfloor \frac{n}{m+1} \right\rfloor$ and we define $$\sigma=\bigcup_{j=0}^{k-1} \{1+j(m+1), 2+j(m+1),\ldots, m-1+j(m+1)\}.$$ 
We consider two cases.

(a) If $n=(k+1)(m+1)-1$ or $n=(k+1)(m+1)-2$, let $\tau =\sigma\cup \{k(m+1)+1, \linebreak k(m+1)+2,\ldots,k(m+1)+m-1\}$. Note that $|\tau|= (k+1)(m-1)$ and $\mathcal P_{d,\tau}=\emptyset$, for $d = |\tau|+1$. 
Indeed, $u=\prod_{j\in\tau}x_j \notin I_{n,m}$, but $x_iu\in I_{n,m}$ for all $i\notin \tau$.

(b) If $n$ is not as in the case (a), let $\tau =\sigma\cup \{k(m+1),\ldots,n\}$. Note that $n-|\tau|=2k-1$ and $\mathcal P_{d,\tau}=\emptyset$, for $d = |\tau|+1$. Indeed, $u=\prod_{j\in\tau}x_j \notin I_{n,m}$, but $x_iu\in I_{n,m}$ for all $i\notin \tau$.

Therefore $\sdepth(S/I_{n,m}) \leq |\tau|$, in both cases. On the other hand, one can easily check that $|\tau|= n+1 - \left\lfloor \frac{n+1}{m+1} \right\rfloor - \left\lceil \frac{n+1}{m+1} \right\rceil$.  Therefore  $\sdepth(S/I_{n,m})\leq \varphi(n,m)$.

\begin{obs}
\emph{One possible way to generalize Theorem $1.3$ and \cite[Theorem 6]{alin}, in the same time, would be to prove that $\sdepth(S/I_{n,m}^k) = \depth(S/I_{n,m}^k)$ for any $k\geq 1$. Furthermore, we might conjecture that if $\Delta$ is a simplicial tree, then $\sdepth(S/I(\Delta)^k) = \depth(S/I(\Delta)^k)$ for any $k\geq 1$.}
\end{obs}

\vspace{2mm} \noindent {\footnotesize
\begin{minipage}[b]{15cm}
Mircea Cimpoea\c s, Simion Stoilow Institute of Mathematics, Research unit 5, P.O.Box 1-764,\\
Bucharest 014700, Romania\\
E-mail: mircea.cimpoeas@imar.ro
\end{minipage}}
\end{document}